\documentclass[
a4paper,
11pt,
DIV=14,
toc=bibliography, 
headsepline, 
parskip=half-,
abstract=true,
]{scrartcl}

\usepackage[centertags]{amsmath}
\usepackage[utf8]{inputenc}
\usepackage{enumerate,amsfonts,amssymb,amsthm,amsopn,cite,array,comment}

\usepackage[usenames]{color}
\definecolor{citegreen}{rgb}{0,0.6,0}
\definecolor{refred}{rgb}{0.8,0,0}
\usepackage[colorlinks, citecolor=citegreen, linkcolor=refred, urlcolor=black]{hyperref}

\title{A Local Singularity Analysis for the Ricci Flow and its Applications to Ricci Flows with Bounded Scalar Curvature -- Part II}
\author{Reto Buzano and Gianmichele Di Matteo}
\date{}

\newtheorem{theorem}{Theorem}[section]
\newtheorem{lemma}[theorem]{Lemma}
\newtheorem{corollary}[theorem]{Corollary}
\newtheorem{proposition}[theorem]{Proposition}

\theoremstyle{definition}
\newtheorem{remark}[theorem]{Remark}

\newtheorem{definition}[theorem]{Definition}
\numberwithin{equation}{section}

\newcommand{\mc}{\mathcal}

\providecommand{\set}[1]{\{ {#1} \}}
\providecommand{\abs}[1]{\lvert #1\rvert}
\providecommand{\norm}[1]{\lVert #1\rVert}

\DeclareMathOperator{\supp}{supp}
\DeclareMathOperator{\diam}{diam}

\DeclareMathOperator{\Sc}{R}
\DeclareMathOperator{\Ric}{Ric}
\DeclareMathOperator{\Rm}{Rm}

\DeclareMathOperator{\Var}{Var}

\newcommand\printaddress{{
\setlength{\parindent}{15pt}
\footnotesize~
\par
{\scshape Reto Buzano}
\newline 
Universit\`a di Torino, 
Dipartimento di Matematica,
Via Carlo Alberto 10, 
10123 Torino, Italy 
\newline
\emph{E-mail address:} 
\texttt{reto.buzano@unito.it}
\par
{\scshape Gianmichele Di Matteo}
\newline 
Scuola Superiore Meridionale, 
Largo San Marcellino 10, 
80138 Napoli, Italy
\newline
\emph{E-mail address:} 
\texttt{g.dimatteo@ssmeridionale.it}
\par
}}

\begin{document}
\maketitle
\begin{abstract}
We continue our local singularity analysis for Ricci flow initiated in \cite{BuDM}. Building on that framework, we study Type~I singular points in general Ricci flows, without assuming any global Type~I curvature bound, and prove that the scalar curvature must blow up at a Type~I rate at each such point in all dimensions. As a consequence, Ricci flows with bounded scalar curvature cannot develop Type~I singular points. This extends earlier results of the first author with Enders and Topping \cite{EMT11} and with Mantegazza \cite{MM15} that relied on a global Type~I assumption. We then adapt the same local perspective to ancient Ricci flows and analyse the curvature behaviour as $t \to -\infty$, showing in particular that every ancient Type~I point exhibits scalar curvature behaviour of ancient Type~I order.
\end{abstract}


\section{Introduction and Main Results}

\subsection{Finite Time Singularities}

In this article, a Ricci flow on an interval $I \subseteq \mathbb{R}$ is a smooth family $(M^n,g(t))_{t\in I}$ of complete $n$-dimensional Riemannian manifolds evolving by $\partial_t g(t) = -2\Ric_{g(t)}$. If $M$ is not closed, we will always assume that it has bounded curvature on compact time-intervals, i.e. 
\begin{equation}\label{eq.bddcurv}
\forall J \subseteq I \text{ compact }\, \exists C<\infty \text{ such that } \abs{\Rm} \leq C \text{ on } M \times J.
\end{equation}

A Ricci flow on $I=[0,T)$ develops a singularity at time $T<\infty$ if it cannot be smoothly extended past $T$. If $M$ is closed or satisfies \eqref{eq.bddcurv}, by Hamilton \cite{Ham82} and Shi \cite{Shi89}, this is equivalent to
\begin{equation}\label{eq.curvatureblowup}
\limsup_{t\nearrow T}\,\sup_M\abs{\Rm(\cdot,t)}_{g(t)}=\infty
\end{equation}
and by Sesum \cite{Ses05}, it is in fact equivalent to $\limsup_{t\nearrow T}\,\sup_M\abs{\Ric(\cdot,t)}_{g(t)}=\infty$.

Hamilton \cite{Ham95} introduced two different \emph{global} notions of finite time singularities, called Type~I and Type~II singularities, by considering the rate at which the curvature blow-up in \eqref{eq.curvatureblowup} happens. In \cite{BuDM}, we introduced \emph{local} versions of these notions, studying blow-up rates locally around given points, see Definition \ref{def.singpoints} below. In the first part of this article, we continue this analysis by investigating Type~I singular points inside Ricci flows satisfying \eqref{eq.bddcurv} and \eqref{eq.curvatureblowup}, but no global Type~I curvature bound.

It was proved by the first author with Enders and Topping \cite{EMT11}, see also \cite{MM15} for an alternative proof, that under a global Type~I assumption, at every singular point one can take a smooth blow-up limit to a smooth and non-trivial Ricci shrinker. Since these non-trivial Ricci shrinkers have positive scalar curvature, a direct corollary of this is that the scalar curvature of the original Type~I flow must blow-up at a Type~I rate as well. It is natural to expect that this result also holds for Type~I singular points in a Ricci flow that does not satisfy a global curvature assumption. In this article, we confirm this expectation.

\begin{theorem}[Scalar Curvature Blow-up at Type~I Singular Points]\label{mainthm}
Let $(M,g(t))_{[0,T)}$ be a complete Ricci flow satisfying \eqref{eq.bddcurv} and \eqref{eq.curvatureblowup}. Let $p \in \Sigma_I$ be a Type~I point in the sense of Definition \ref{def.singpoints} below. Then there exist a constant $c_0>0$ and a time $T_0 \in [0,T)$ such that the scalar curvature satisfies $\Sc(p,t)\geq \tfrac{c_0}{T-t}$ for all $t \in [T_0,T)$.
\end{theorem}

To put Theorem \ref{mainthm} into context, let us collect some important definitions and results from our first article \cite{BuDM}. We start with our definition of Riemann scale and with our local notions of Type~I and Type~II singular points.

\begin{definition}[Parabolic Cylinders, Riemann Scale, and Singular Points]\label{def.Riemannscale}
Let $(M,g(t))_{t\in I}$ be a Ricci flow on $I$ satisfying \eqref{eq.bddcurv} and let $(p,t) \in M \times I$ be a space-time point. For $r>0$, we define the \emph{parabolic cylinder} $\mathcal{P}(p,t,r)$ with center $(p,t)$ and radius $r$ by
\begin{equation*}
\mathcal{P}(p,t,r) := B_{g(t)}(p,r)\times \{ (t-r^2, t+r^2)\cap I \}.
\end{equation*}
We then define the \emph{Riemann scale} $r_{\Rm}(p,t)$ at (p,t) by
\begin{equation*}
r_{\Rm}(p,t) := \sup \{ r>0 \mid \abs{\Rm}<r^{-2} \text{ on } \mathcal{P}(p,t,r) \}.
\end{equation*}
If $(M,g(t))$ is flat for every $t\in I$, we set $r_{\Rm}(p,t)=+\infty$. If $I=[0,T)$, $T<\infty$, and \eqref{eq.curvatureblowup} holds, then we call a point $p \in M$ a \emph{singular point} if 
\begin{equation}\label{eq.singular}
\limsup_{t\nearrow T}\, r_{\Rm}^{-2}(p,t)=\infty.
\end{equation}
We write $\Sigma$ for the set of all singular points.
\end{definition}

\begin{remark} One can show that a point $p \in M$ is a singular point in the above sense if and only if for any neighbourhood $U$ of $p$, the Riemannian curvature becomes unbounded on $U$ as $t \nearrow T$, see Theorem 1.4 in \cite{BuDM}, thus our Definition is equivalent to Definition 1.5 in \cite{EMT11}. Moreover, by Corollary 2.4 in \cite{BuDM}, if $p$ is a singular point, then we can actually strengthen \eqref{eq.singular} to 
\begin{equation}\label{eq.singularstrong}
\liminf_{t\nearrow T} (T-t)r_{\Rm}^{-2}(p,t) > 1.
\end{equation}
It therefore makes sense to consider the following two types of singular points.
\end{remark}

\begin{definition}[Type~I and Type~II Singular Points]\label{def.singpoints}
Let $(M,g(t))_{t\in [0,T)}$ be a Ricci flow with $T<\infty$ satisfying \eqref{eq.bddcurv} and \eqref{eq.curvatureblowup}.
\begin{enumerate}[i)]
\item A point $p \in M$ is a \emph{Type~I singular point} if there exist constants $0<c_I \leq C_I<\infty$ such that
\begin{equation}\label{eq.typeI}
c_I \leq \limsup_{t\nearrow T}\, (T-t) r_{\Rm}^{-2}(p,t) \leq C_I.
\end{equation}
The set of all such points is denoted by $\Sigma_I$.
\item We say that a point $p$ is a \emph{Type~II singular point} if 
\begin{equation}\label{eq.typeII}
\limsup_{t\nearrow T}\, (T-t) r_{\Rm}^{-2}(p,t)=\infty
\end{equation}
and write $\Sigma_{II}$ for the set of all such points.
\end{enumerate}
\end{definition}

\begin{remark} See Definition~1.1 and Theorem~1.4 in \cite{BuDM} for equivalent definitions of Type~I and Type~II points that use the parabolically rescaled Ricci flows $\widetilde{g}_t(s):= (T-t)^{-1}g(t+(T-t)s)$ rather than the Riemann scale. Moreover, let us remark that the parabolic neighbourhoods involved in this Definition may collapse geometrically, hence the Definition captures behaviours between local and pointwise regimes.
\end{remark}

Analogously to Definitions \ref{def.Riemannscale}, we also define the \emph{Ricci scale} $r_{\Ric}(p,t)$ at $(p,t)$ by
\begin{equation*}
r_{\Ric}(p,t) := \sup \{ r>0 \mid \abs{\Ric}<a_0(n)r^{-2} \text{ on } \mathcal{P}(p,t,r) \},
\end{equation*}
where $a_0=a_0(n):=\sqrt{n}(n-1)$. This normalisation guarantees the relation $r_{\Rm}(p,t) \leq r_{\Ric}(p,t)$. If $(M,g(t))$ is Ricci-flat for every $t\in[0,T)$, we set $r_{\Ric}(p,t)=+\infty$. Replacing the Riemann scale with the Ricci scale in \eqref{eq.singular}, \eqref{eq.typeI}, and \eqref{eq.typeII} we also get notions of \emph{Ricci singular points} $\Sigma^{\Ric}$, \emph{Ricci Type~I singular points} $\Sigma_I^{\Ric}$, and \emph{Ricci Type~II singular points} $\Sigma_{II}^{\Ric}$, respectively.

Under a global Type~I assumption, one can show that $\Sigma = \Sigma_I = \Sigma_I^{\Ric}$, see e.g.~Theorem $3.2$ in \cite{EMT11}. In \cite{BuDM}, we proved the following results for general Ricci flows.

\begin{proposition}[Theorems 1.2, 1.9, and Corollary 4.10 in \cite{BuDM}]
Let $(M,g(t))$ be a Ricci flow satisfying \eqref{eq.bddcurv} and \eqref{eq.curvatureblowup}. Then the following hold.
\begin{enumerate}[i)]
\item $\Sigma=\Sigma_I \cup \Sigma_{II}$,
\item $\Sigma^{\Ric} = \Sigma_I^{\Ric} \cup \Sigma_{II}^{\Ric}$,
\item $\Sigma = \Sigma^{\Ric}$.
\end{enumerate}
\end{proposition}

The first point of this proposition implies that for every singular point $p\in\Sigma$, the Riemann tensor blows up at least at a Type~I rate. More precisely, there is an essential blow-up sequence $(p_i,t_i)$ such that one has
\begin{equation}\label{eq.essentialTypeI}
\lim_{i\to\infty}\, (T-t_i)\abs{\Rm(p_i,t_i)}_{g(t_i)} >0 \qquad\text{and}\qquad \limsup_{i\to\infty}\, \frac{d^2_{g(t_i)}(p_i,p)}{T-t_i} <\infty.
\end{equation}
By the third point, which is a local version of Sesum's result \cite{Ses05}, the Ricci curvature blows up at every singular point, and by the second point this Ricci curvature blow-up must again happen at least at a Type~I rate, meaning in particular that in \eqref{eq.essentialTypeI} one can replace the Riemannian by the Ricci curvature. Together, these results show that every Type~I singular point is also a Ricci Type~I singular point, $\Sigma_I\subseteq \Sigma_I^{\Ric}$. By the main theorem of this article, every such point is also a \emph{scalar curvature Type~I singular point}. Theorem~\ref{mainthm} is in fact stronger, showing that the scalar curvature (and hence also the Ricci and Riemann tensors) blows up at a Type~I rate \emph{at the point $p$} and not just ''in a neighbourhood" of $p$. Moreover, this is true for \emph{every} sequence $t_i \nearrow T$.

In particular, with our new Theorem \ref{mainthm}, we can strengthen \eqref{eq.essentialTypeI} for Type~I points as follows: we can pick $t_i \nearrow T$ \emph{arbitrary} and $p_i \equiv p$ \emph{fixed} to obtain an essential blow-up sequence $(p,t_i)$ with
\begin{equation}\label{eq.newessentialTypeI}
\liminf_{i\to\infty}\, (T-t_i)\abs{\Rm(p,t_i)}_{g(t_i)} >0.
\end{equation}

Finally, we also recall the definition of well behaved (singular) points from \cite{BuDM}. These are points $q$ where we do not only have the obvious bound $\abs{\Ric}(q,t) \leq a_0 r^{-2}_{\Ric}(q,t)$, but also a reverse estimate for times $t$ close to the singular time.

\begin{definition}[Well Behaved Singular Points]\label{defn.wellbehaved}
Let us consider a Ricci flow $(M,g(t))$ defined on $[0,T)$, and let us assume that $\abs{\Ric}$ is not identically $0$. For $\delta \in(0,1)$ and $t_1 \in [0,T)$, we define the set of \emph{$\delta$-well behaved points}
\begin{equation}\label{eq.wellbehaved}
G_\delta = G_{\delta,t_1}:= \big\{ q \in M \mid \delta a_0 r^{-2}_{\Ric}(q,t) < \abs{\Ric(q,t)}_{g(t)}, \text{ for all }t\in[t_1,T)\big\}.
\end{equation}
The set $\Sigma_{\delta} := \Sigma \cap G_\delta$ is the set of points that are both singular and well behaved.
\end{definition}

Under a uniform scalar curvature bound, we then proved a non-existence result for well-behaved singular points in dimensions $n<8$ and showed that in higher dimensions the set $\Sigma_{\delta}$ has codimension at least $8$ in a suitable sense (see Theorems 1.11 and 1.13 in \cite{BuDM}). Implicit in these results was the fact that $\Sigma_{\delta} \subseteq \Sigma_{II}$ for any $\delta>0$, so these results rule out certain Type~II singular points in dimensions $n<8$ and show that they have codimension at least $8$ in higher dimensions in a bounded scalar curvature Ricci flow. The main theorem of this article complements these results, as it rules out Type~I singular points in a bounded scalar curvature Ricci flow, in fact in every dimension $n$. In particular, we obtain the following corollary to Theorem \ref{mainthm}.

\begin{corollary}[Only Badly Behaved Type II Points in Bounded Scalar Curvature Flows]
Let $(M,g(t))_{[0,T)}$, $T<\infty$ be a closed Ricci flow of dimension $n < 8$ with uniformly bounded scalar curvature $\abs{\Sc}\leq R_0$ on $M\times [0,T)$. If $T$ is a singular time, i.e. if \eqref{eq.curvatureblowup} holds, then $\Sigma = \Sigma_{II}$ and for each $p\in \Sigma$ and each sequence $\delta_i \searrow 0$, there are essential blow-up sequences $(p_i,t_i)$ with $p_i \to p$ and $t_i \nearrow T$ such that $p_i \not\in G_{\delta_i,t_i}$, or in other words
\begin{equation}
\lim_{i \to \infty} r_{\Ric}^2(p_i,t_i)\abs{\Ric(p_i,t_i)}_{g(t_i)}=0.
\end{equation}
\end{corollary} 

\subsection{Ancient Ricci Flows}\label{subsec.ancient}

In the second part of the article, we study ancient Ricci flows, i.e. Ricci flows $(M,g(t))_{t\in I}$ defined for $t\in I=(-\infty,0]$, still assuming the bounded curvature on compact time intervals assumption \eqref{eq.bddcurv}. Also in this case, Hamilton \cite{Ham95} introduced two \emph{global} notions of flows, called Type~I and Type~II ancient solutions. Here, we want to adapt our \emph{local} perspective to distinguish points in ancient Ricci flows in terms of their curvature behaviour as $t \to -\infty$. 

Defining the Riemann scale as in Definition \ref{def.Riemannscale}, we introduce the following notion of ancient Type~I and Type~II points, similar to Definition \ref{def.singpoints}.

\begin{definition}[Ancient Type~I and Type~II Points]\label{def.ancientpoints}
Let $(M,g(t))_{t\in (-\infty,0]}$ be a complete Ricci flow satisfying \eqref{eq.bddcurv}.
\begin{enumerate}[i)]
\item A point $p \in M$ is an \emph{ancient Type~I point} if there exist constants $0<c_I \leq C_I<\infty$ such that
\begin{equation}\label{eq.ancienttypeI}
c_I \leq \limsup_{t \to -\infty} |t| r_{\Rm}^{-2}(p,t) \leq C_I.
\end{equation}
The set of all such points is denoted by $A_I$.
\item We say that a point $p$ is an \emph{ancient Type~II point} if 
\begin{equation}\label{eq.ancienttypeII}
\limsup_{t\to -\infty}\, |t| r_{\Rm}^{-2}(p,t)=\infty
\end{equation}
and write $A_{II}$ for the set of all such points.
\end{enumerate}
As before, using the Ricci scale rather than the Riemann scale in \eqref{eq.ancienttypeI} and \eqref{eq.ancienttypeII}, we also get notions of \emph{ancient Ricci Type~I points} $A_I^{\Ric}$, and \emph{ancient Ricci Type~II points} $A_{II}^{\Ric}$, respectively.
\end{definition}

As a first result, we show that in a non-flat ancient Ricci flow, every point $p\in M$ falls in one of the above two categories, meaning that the curvature tensor cannot go to zero too fast as $t\to-\infty$. In fact, we have a lower bound for $r_{\Rm}^{-2}(p,t)$ similar to \eqref{eq.singularstrong}.

\begin{proposition}[Every point in a non-flat ancient Ricci flow is of Type~I or Type~II]\label{prop.ancient}
Suppose $(M,g(t))_{t\in (-\infty,0]}$ is a complete non-flat Ricci flow satisfying \eqref{eq.bddcurv}. Then for every point $p\in M$ we have 
\begin{equation}\label{eq.strongancient}
\liminf_{t \to -\infty} |t| r_{\Rm}^{-2}(p,t)>1. 
\end{equation}
In particular, we conclude that $M= A_I \cup A_{II}$. Similarly, if we assume additionally that the flow is not Ricci-flat (i.e. not static), then we get the same lower bound for the Ricci scale and hence $M= A^{\Ric}_I \cup A^{\Ric}_{II}$.
\end{proposition}

Note that for a non-flat gradient shrinking Ricci soliton with bounded curvature, every point is Type~I, while for a non-flat gradient steady Ricci soliton every point is Type~II. In ancient flows such as the King-Rosenau solution or Perelman's ovals, both ancient Type~I and ancient Type~II points exist.

Our second main result of this article, analogous to Theorem \ref{mainthm}, is a scalar curvature bound for ancient Type~I points.

\begin{theorem}[Scalar Curvature Behaviour at Ancient Type~I Points]\label{ancientthm}
Let $(M,g(t))_{(-\infty,0)}$ be a non-flat, complete, ancient Ricci flow satisfying \eqref{eq.bddcurv}. Let $p \in A_I$ be an ancient Type~I point in the sense of Definition \ref{def.ancientpoints} above. Then there exist a constant $c_0>0$ and a time $T_0 \in (-\infty,0]$ such that $\Sc(p,t)\geq \tfrac{c_0}{\abs{t}}$ for all $t \in (-\infty,T_0]$.
\end{theorem}

The main difficulties in proving Theorems~\ref{mainthm} and \ref{ancientthm} come from the facts that, in general, one cannot take a smooth global blow-up (or blow-down) limit without a global curvature assumption, and working only with a local limit (which one can always obtain at a Type~I point) does not seem to give sufficient information to conclude positivity of the scalar curvature of the limit. We therefore need to work with a weak global limit, or more precisely with the notion of $\mathbb{F}$-convergence to a singular Ricci shrinker in the sense of Bamler \cite{Bamler2, Bamler3}. In Section~\ref{sec.Bamler}, we give a short overview of this theory and of the main blow-up and blow-down results. We then prove Theorem~\ref{mainthm} in Section~\ref{sec.blowup}. Finally, we prove Proposition~\ref{prop.ancient} and Theorem~\ref{ancientthm} in Section~\ref{sec.ancient}.\\

\textbf{Acknowledgements.} The authors thank Davide Dameno for interesting discussions. RB has been partially supported by the Italian PRIN project ``Differential-geometric aspects of manifolds via global analysis'' (No.~20225J97H5).

\section{Bamler's Theory in a Nutshell}\label{sec.Bamler}

In the groundbreaking trilogy \cite{Bamler1, Bamler2, Bamler3}, Bamler developed a weak compactness theory for Ricci flows in arbitrary dimensions, as well as a precise structure theory for non-collapsed weak limits obtained this way. While originally phrased for closed manifolds, in the appendix of \cite{Bamler4} (see also \cite{CMZ21, MZ21}), he justified why his results also work for complete non-compact Ricci flows with bounded curvature on compact time intervals, which is the setting we are interested in here. Let us remark that in the special case of Ricci flows given by gradient shrinking Ricci solitons, Li-Wang \cite{LW1} and Bertellotti-Buzano \cite{BB25} have further extended some of Bamler's theory to the case without bounded curvature.

In this article, we use Bamler's theory as a black box. For this reason, we give a brief summary of the fundamental definitions and results concerning weak blow-ups at the first singular time and blow-downs at negative infinity, as well as the regular-singular decomposition of the limits obtained this way. These results are needed to prove Theorems~\ref{mainthm} and ~\ref{ancientthm} in the next sections. For more detailed summaries of these parts of Bamler's theory, we refer the reader for example to the appendix of Li-Wang \cite{LW2}, and Chapter 1 of Bertellotti's PhD thesis \cite{Ber-thesis} -- but for many subtle, technical details, it is probably best to go back to Bamler's original work.

A fundamental ingredient in Bamler's theory is his weak compactness theory for so-called metric flow pairs. To explain this concept, let us start with a \emph{rough} version of the definition of metric flows.
\begin{definition}
A \emph{metric flow} over a subset $I \subseteq \mathbb{R}$ is a tuple
\begin{equation}
\big(\mathcal{X}, \mathfrak{t}, (d_t)_{t\in I}, (\nu_{\mathbf{x};s})_{\mathbf{x}\in \mathcal{X},s\in I, s\leq \mathfrak{t}(\mathbf{x})}\big)
\end{equation}
where $\mathcal{X}$ is a set of points, $\mathfrak{t}:\mathcal{X}\to I$ is a map called time function, $d_t$ are complete and separable metrics on the time slices $\mathcal{X}_t := \mathfrak{t}^{-1}(t)$, and $\nu_{\mathbf{x};s}$ are probability measures satisfying the reproduction formula
\begin{equation}
\nu_{\mathbf{x};t_1}(A)=\int_{\mathcal{X}_{t_2} } \nu_{\mathbf{y};{t_1}}(A) d\nu_{\mathbf{x};t_2}(\mathbf{y}), \quad \forall t_1,t_2,t_3 \in I, \ t_1 \leq t_2 \leq t_3, \ \forall \mathbf{x} \in \mathcal{X}_{t_3},
\end{equation}
for every Borel set $A \subseteq \mathcal{X}_{t_1}$, as well as $\nu_{\mathbf{x};\mathfrak{t}(\mathbf{x})}=\delta_\mathbf{x}$ for all $\mathbf{x}\in\mathcal{X}$. The measures $(\nu_{\mathbf{x};s})_{s\in I, s\leq \mathfrak{t}(\mathbf{x})}$ are called \emph{conjugate heat kernels} at $\mathbf{x}$. We abbreviate a metric flow simply by $\mathcal{X}$.
\end{definition}

For the \emph{precise} definition (including in particular also a technical regularity requirement), we refer the reader to Definition 3.1 in \cite{Bamler2}. We remark that all the results of Bamler stated below work with this precise definition.

\begin{definition} 
A metric flow is \emph{$H$-concentrated} if for any $s\leq t$ in $I$ and $\mathbf{x},\mathbf{y}\in \mathcal{X}_t$, we have
\begin{equation}
\Var(\nu_{\mathbf{x};s},\nu_{\mathbf{y};s}) := \int_{\mathcal{X}_s}\int_{\mathcal{X}_s} d_s^2(\mathbf{z}_1,\mathbf{z}_2) d\nu_{\mathbf{x};s}(\mathbf{z}_1) d\nu_{\mathbf{y};s}(\mathbf{z}_2) \leq d^2_t(\mathbf{x},\mathbf{y})+ H(t-s).
\end{equation}
A point $\mathbf{z} \in \mathcal{X}_s$ is called an \emph{$H$-center} of $\mathbf{x}\in\mathcal{X}_t$ if $s\leq t$ and
\begin{equation}
\Var(\nu_{\mathbf{x};s},\delta_\mathbf{z})  \leq H(t-s).
\end{equation}
\end{definition}

By Proposition 3.25 of \cite{Bamler2}, if $\mathcal{X}$ is $H$-concentrated, then for every $\mathbf{x}\in \mathcal{X}_t$ and every $s\in I$ with $s\leq t$, there exist $H$-centers $\mathbf{z} \in \mathcal{X}_s$ of $\mathbf{x}$. These will play a crucial role below.

\begin{definition} 
A family of probability measures $(\mu_t)_{t\in I'}$, where $I'\subseteq I$, is called a \emph{conjugate heat flow} if for any $s,t\in I'$ with $s\leq t$ we have
\begin{equation}
\mu_s(A)=\int_{\mathcal{X}_t} \nu_{\mathbf{x};s}(A) d\mu_t(\mathbf{x}),
\end{equation}
for every Borel set $A \subseteq \mathcal{X}_s$. Finally, for an interval $I \subseteq \mathbb{R}$, a \emph{metric flow pair} is a pair $(\mathcal{X},(\mu_t)_{t\in I'})$ where $I'\subseteq I$ with $|I \setminus I' |=0$, $\mathcal{X}$ is a metric flow over $I'$, and $(\mu_t)_{t\in I'}$ is a conjugate heat flow that has full support for all $t\in I'$.
\end{definition}

Now consider for $I=[0,T)$ a complete $n$-dimensional Ricci flow $(M, g(t))_{t\in I}$ satisfying \eqref{eq.bddcurv} and \eqref{eq.curvatureblowup}. We can associate to this Ricci flow a metric flow over $I$ in a natural way. First, let
\begin{equation*}
\mathcal{X}=M\times I
\end{equation*}
and denote by $\mathfrak{t} : \mathcal{X} \to I$ the projection onto the second factor. Then each time slice $\mathcal{X}_{t}:=\mathfrak{t}^{-1}(t)$ is given by $\mathcal{X}_{t}=M\times\{t\}\simeq M$, and we endow $\mathcal{X}_{t}$ with the Riemannian distance $d_{t}=d_{g(t)}$, turning it into a complete metric space $(\mathcal{X}_{t},d_{t})$. Next, for $(y,s) \in M\times I$ consider the heat kernel $K(\cdot,\cdot ; y,s)$, i.e.~the positive smooth solution of 
\begin{equation*}
\partial_t K(\cdot,t;y,s) = \Delta_{g(t)} K(\cdot,t;y,s)
\end{equation*}
satisfying 
\begin{equation*}
\lim_{t\searrow s}K(\cdot,t;y,s) = \delta_y
\end{equation*} 
in the weak* sense. Fixing instead $(x,t)\in M\times I$, then $K(x,t; \cdot,\cdot)$ satisfies the conjugate heat equation 
\begin{equation*}
\partial_s K(x,t; \cdot,s) = -\Delta_{g(s)} K(x,t; \cdot,s) + \Sc_{g(s)} K(x,t; \cdot,s)
\end{equation*}
as well as
\begin{equation*}
\lim_{s\nearrow t} K(x,t; \cdot,s) = \delta_x
\end{equation*} 
in the weak* sense. For a point $\mathbf{x}=(x,t) \in \mathcal{X}$, we then set
\begin{equation*}
\nu_{\mathbf{x};s} = \begin{cases} 
K(x,t; \cdot,s) dV_{g(s)}, & \text{if } s<t,\\
\delta_\mathbf{x}, & \text{if } s=t.
\end{cases}
\end{equation*}

\begin{proposition}[Theorem 3.36 of \cite{Bamler2}]\label{RFmetricflow}
$\mathcal{X} = (\mathcal{X}, \mathfrak{t}, (d_t)_{t\in I}, (\nu_{\mathbf{x};s})_{\mathbf{x}\in \mathcal{X},s\in I, s\leq \mathfrak{t}(\mathbf{x})})$ obtained as above is an $H_n$-concentrated metric flow, where 
\begin{equation}\label{defHn}
H_n := \frac{n-1}{2}\pi^2 + 4.
\end{equation}
Moreover, conjugate heat flows on this $\mathcal{X}$ are precisely the probability measures $\mu_t = v(\cdot,t) dV_{g(t)}$ where $v$ is a nonnegative solution of the conjugate heat equation.
\end{proposition}

In the following, we write $\Var(\mu_t) := \Var(\mu_t, \mu_t)$. In Section 2.6 of \cite{Bamler3}, Bamler considers tangent flows at the first singular time $T$ of a Ricci flow as follows: first, he constructs a final time-slice for the given flow defined by
\begin{equation*}
M_T := \{ (\mu_{t})_{t\in [0,T)} \mid \mu_t \text{ is a conjugate heat flow on $(M,g(t))$ with } \lim_{t \nearrow T} \Var(\mu_t)=0\}.
\end{equation*} 
endowed with the distance function
\begin{equation*}
d_{M_T}\big( (\mu_{1,t})_{[0,T)}),(\mu_{2,t})_{[0,T)} \big):=\lim_{t \nearrow T} d_{W_1}^{g_t}(\mu_{1,t},\mu_{2,t}) \in [0,+\infty].
\end{equation*}
Note that the limit exists by the monotonicity of the $W_1$-Wasserstein distance \cite[Lemma 2.7]{Bamler1}. By Lemma 2.35 of \cite{Bamler3}, $(M_T,d_{M_T})$ is a complete metric space (allowing infinite distances), locally isometric to the (possibly degenerate) metric space induced by $g_T:=\lim_{t \nearrow T} g_t$, on each open set where this limit exists smoothly.

Denoting elements in $M_T$ as points $x=(\mu_{t})_{t\in[0,T)}$, and setting 
\begin{equation*}
d\nu_{(x,T);s} =d \mu_s=K(x,T; \cdot, s) dV_{g(s)} =(4 \pi (T-s))^{-\frac{n}{2}} e^{-f(\cdot,s)} dV_{g(s)}
\end{equation*}
this gives a definition of a conjugate heat kernel $K(x,T; \cdot, \cdot)$ based at $(x,T)$.

\begin{theorem}[Theorem 2.37 of \cite{Bamler3}]\label{thm.Bamlerblowup1}
Take any $\mu_t =\nu_{(x,T);t} \in M_T$ as above and consider the rescaled flows
\begin{equation*}
g^i(t):=\lambda_i g(\lambda_i^{-1} t+T), \qquad\qquad t\in [-\lambda_i T,0),
\end{equation*}
for some choice of rescaling parameters $\lambda_i \nearrow \infty$, as well as the correspondingly rescaled conjugate heat flow measures
\begin{equation*}
{\mu}^i_t:={\nu}_{(x,T);\lambda_i^{-1}t+T}, \qquad\qquad t\in [-\lambda_i T,0).
\end{equation*}
Interpreting $(M,(g^i(t))_{t\in[-\lambda_i T,0)},({\mu}^i_t)_{t\in[-\lambda_i T,0)})$ as metric flow pairs (as in Proposition \ref{RFmetricflow} above), there exists a subsequence for which we have uniform $\mathbb{F}$-convergence on compact time intervals within some correspondence $\mathfrak{C}$ to a limit metric flow pair $(\mathcal{X}^\infty,(\mu^\infty_t)_{t\in (-\infty,0)})$ that is a metric soliton. Moreover, the Nash entropy along the conjugate heat flow $\mu^\infty_t$ is constant.
\end{theorem}

For the definitions of metric solitons and of $\mathbb{F}$-convergence within a correspondence $\mathfrak{C}$, see Sections 3 and 6 of \cite{Bamler2}. Luckily, for most purposes, we will not need these concepts at all, due to the following further structural results.

\begin{theorem}[Bamler \cite{Bamler3}, in particular Theorems 2.4 and 2.18]\label{thm.Bamlerblowup2}
We can identify $\mathcal{X}^\infty$ with a product metric flow $X \times (-\infty,0)$, where $(X,d)$ is a singular space for some metric $d$ and $(\mathcal{X}_t^\infty,d_t)=(X \times \set{t},\sqrt{|t|}d)$.
Moreover, $\mathcal{X}^\infty$ can be decomposed into an open and dense regular part $\mathcal{R}$ and a singular part $\mathcal{S}$ of measure zero. On the entire regular part $\mathcal{R}$ the $\mathbb{F}$-convergence is actually in the smooth sense and we can define a limit metric $g^{\infty}$ on $\mathcal{R}$ (see the remark below). We then have the identifications
\begin{equation}
\mathcal{R}=\mathcal{R}_X \times (-\infty,0) \qquad\text{and}\qquad (\mathcal{R}_t,g_t^\infty)=(\mathcal{R}_X \times \set{t},|t| g_X).
\end{equation}
Furthermore, writing 
\begin{equation*}
\mu^\infty_t = \nu_{(x^\infty,0);t }= K(x^\infty,0;\cdot,t) dV_{g_t^\infty} =(-4 \pi t)^{-\frac{n}{2}} e^{-f(\cdot,t)} dV_{g_t^\infty}
\end{equation*}
on $\mathcal{R}$, then the gradient Ricci soliton equation
\begin{equation}
\Ric_{g_t^\infty} + \nabla^2 f = \frac{1}{2\abs{t}}g_t^\infty
\end{equation}
holds on the regular part $\mathcal{R}$. Finally, we have non-negative scalar curvature $\Sc\geq 0$ on $\mathcal{R}_X$ and $(X,d)$ is either a metric cone or $\Sc>0$ on $\mathcal{R}_X$.
\end{theorem} 

\begin{remark}\label{remark.smoothconv}
In the preceding theorem, smooth convergence on $\mathcal{R}$ means in particular that there exists an increasing sequence $(U_i)_{i\in \mathbb{N}}$ of open subsets of $\mathcal{R}$ with $\bigcup_{i=1}^{\infty} U_i=\mathcal{R}$, open subsets $V_i \subset M \times I$, time-preserving diffeomorphisms $\phi_i: U_i \to V_i$, and a sequence $\varepsilon_i \searrow 0$ such that
\begin{equation*}
\norm{\phi_i^* g^i-g^{\infty}}_{C^{[\varepsilon_i^{-1}]}(U_i)} \leq \varepsilon_i,
\end{equation*}
where $g^i$ is the spacetime metric induced by $g^i(t)$. This implies classical Cheeger-Gromov convergence. Moreover, there are similar convergence results for the time-direction vector fields $\partial^i_\mathfrak{t}$ and the heat kernels $K^i$. See \cite[Theorem 9.21]{Bamler2} for more details.
\end{remark}

To finish this section, let us also recall the existence and structural result for blow-down limits at negative infinity (also called tangent flows at infinity) from Bamler \cite{Bamler3}, adapted here to the simpler case of a metric flow $\mathcal{X}$ induced by a fixed ancient Ricci flow. 

\begin{theorem}[Theorem 2.40 of \cite{Bamler3}]\label{thm.Bamlerinfinity}
Consider a metric flow $(\mathcal{X},\mathfrak{t},(d_t)_{t \in I}, (\nu_{\mathbf{x};s})_{\mathbf{x} \in \mathcal{X},s \in I, s \leq \mathfrak{t}(\mathbf{x}) } )$ associated to an ancient Ricci flow $(M,g(t))_{t\in(-\infty,0]}$. Fix a point $\mathbf{x}\in\mathcal{X}_0$ and consider the corresponding conjugate heat flow measures $(\nu_{\mathbf{x};s})_{s \leq 0}$. Define for a sequence of positive parameters $\lambda_i \rightarrow 0$, the rescaled flows
\begin{equation*}
g^i(t):=\lambda_i g(\lambda_i^{-1} t), \qquad\qquad t\in (-\infty,0],
\end{equation*}
as well as the correspondingly rescaled conjugate heat flow measures
\begin{equation*}
{\mu}^i_t:={\nu}_{\mathbf{x};\lambda_i^{-1}t}, \qquad\qquad t\in (-\infty,0].
\end{equation*}
Interpreting $(M,(g^i(t))_{t\in(-\infty,0]},({\mu}^i_t)_{t\in(-\infty,0]})$ as metric flow pairs, there exists a subsequence for which we have uniform $\mathbb{F}$-convergence on compact time intervals within a correspondence $\mathfrak{C}$ to a limit metric flow pair $(\mathcal{X}^\infty,(\mu^\infty_t)_{t\in (-\infty,0]})$ that is a metric soliton. Moreover, the Nash entropy along the conjugate heat flow $\mu^\infty_t$ is constant.
\end{theorem}

Note that also in this case Theorem \ref{thm.Bamlerblowup2} holds for the tangent flows at infinity $\mathcal{X}^\infty_{<0}$.

\section{Tangent Flows at Type~I Singular Points}\label{sec.blowup}

In the last section, we have seen Bamler's abstract construction of conjugate heat kernels based at the singular time $T$. Another definition of conjugate heat kernels based at the singular time (at a point $x \in M$) is given in the work of the first author and Mantegazza \cite{MM15}. It is obtained as a $C^{\infty}_{\mathrm{loc}}(M\times [0,T))$ limit of conjugate heat kernels $K(x,s_j; \cdot, \cdot)$ based at regular times $s_j \nearrow T$. These conjugate heat kernels based at $s_j$ exist by \cite[Chapter 24]{Chow10}, and a limit can be taken by a standard Schauder estimate argument, compare for example with Proposition 2.5 in \cite{BB25}. Note that the limit may depend on the chosen sequence $(s_j)$.

Assuming for the moment that such a sequence has been fixed, we write the limit as $\widetilde{K}(x,T; \cdot, \cdot)$, inducing the family of probability measures $\widetilde{\nu}_{(x,T);t}$ in the usual manner. We denote by $\widetilde{M}_T$ the set of all probability measures obtained this way. The following lemma will allow us to apply Bamler's results from Theorems~\ref{thm.Bamlerblowup1} and \ref{thm.Bamlerblowup2} to conjugate heat kernels in $\widetilde{M}_T$.

\begin{lemma}
With the above definitions, we have $\widetilde{M}_T \subseteq M_T$.
\end{lemma}

\begin{proof}
If $M$ is \emph{closed}, this follows quite directly. Indeed, in this case, take a fixed element $\widetilde{\nu}_{(x,T);t} \in \widetilde{M}_T$ coming from a sequence $s_j\nearrow T$. Then for fixed $t\in[0,T)$, an arbitrary point $y_0\in M$, and an arbitrary $1$-Lipschitz function $h$ (with respect to $(M,g(t))$), we obtain (compare with Lemma~3.1 of Hallgren \cite{Hal2})
\begin{align*}
\left|\int_M h \, d\nu_{(x,s_j);t} - \int_M h \, d \widetilde{\nu}_{(x,T);t}\right|
\leq \diam(M,g(t)) \int_M \bigl|K(x,s_j;y, t)-\widetilde{K}(x,T; y, t)\bigr| \, d \mu_{g(t)}(y).
\end{align*}
The right-hand side converges to zero as $j\to\infty$ by Lebesgue's theorem. By Kantorovich duality, after normalising the $1$-Lipschitz test functions (e.g. by $h(y_0)=0$ for some $y_0 \in M$), the previous estimate implies $W_1$-Wasserstein convergence
\begin{equation*}
d_{W_1}^{g(t)}\bigl(\nu_{(x,s_j);t},\widetilde{\nu}_{(x,T);t}\bigr) \to 0.
\end{equation*}
By \cite[Corollary 3.8]{Bamler1}, we have
\begin{equation}\label{eq.variancekernel}
\Var(\nu_{(x,s_j);t}) \leq H_n(s_j-t) \leq H_n(T-t)
\end{equation} 
for $H_n$ as in \eqref{defHn}, and the variance is lower semicontinuous under $W_1$-Wasserstein convergence, hence $\Var(\widetilde{\nu}_{(x,T);t})\leq H_n(T-t)$. In particular $\widetilde{\nu}_{(x,T);t} \in M_T$.

If $M$ is complete and non-compact, then \eqref{eq.variancekernel} still holds, but the above calculation showing $W_1$-Wasserstein convergence no longer works directly. However, if $h$ is a continuous function \emph{with compact support}, an obvious modification of that calculation yields
\begin{equation}\label{eq.weakconvergence}
\int_M h \, d\nu_{(x,s_j);t} - \int_M h \, d \widetilde{\nu}_{(x,T);t} \to 0,
\end{equation}
which shows weak* convergence. Since all measures involved are probability measures, this is equivalent to narrow convergence (i.e. \eqref{eq.weakconvergence} holds for any bounded continuous function). Using the additional uniform bound on the variance from \eqref{eq.variancekernel}, this implies the required $W_1$-Wasserstein convergence by Villani \cite[Theorem 7.12]{Vil03} (see also \cite[Lemma 2.9]{Bamler2}), finishing the proof also in this case.
\end{proof}

The following lemma allows us to recover pointed Gromov--Hausdorff convergence based at $H_n$-centers (or more generally $H$-centers for $H \geq H_n$). It is inspired by \cite{LW2} (see in particular Proposition~A.14), where the analogous result is obtained for metric flows induced by Ricci shrinkers, and the proof in our case follows almost the same way. 

\begin{lemma}[Pointed Gromov--Hausdorff convergence with respect to $H$-centers]\label{GH-conv}
In the situation of Theorem~\ref{thm.Bamlerblowup1} or \ref{thm.Bamlerinfinity}, for any $t \in (-\infty,0)$ at which the $\mathbb{F}$-convergence is time-wise (see Definition~6.1 in \cite{Bamler2}), and for any choice of $H_n$-centers $z_t^i \in \mathcal{X}_t^i$ of the conjugate heat flow measures $\mu_t^i$, there exists a point $z_t^\infty \in \mathcal{X}_t^\infty$ such that, up to taking a subsequence, we have pointed Gromov--Hausdorff convergence of $(\mathcal{X}_t^i,d_t^i,z_t^i)$ to $(\mathcal{X}_t^\infty,d_t^\infty,z_t^\infty)$. Moreover, the same is true for $H$-centers $\widetilde{z}_t^i$ of $\mu_t^i$ for any $H\geq H_n$.
\end{lemma}

\begin{proof}
Let $\mathfrak C$ be the correspondence given by Theorem~\ref{thm.Bamlerblowup1} or \ref{thm.Bamlerinfinity}, and for the fixed time $t$ write
\begin{equation*}
\phi_t^i: (\mc X_t^i,d_t^i)\to (Z_t,d_{Z_t}), \qquad \phi_t^\infty:(\mc X_t^\infty,d_t^\infty)\to (Z_t,d_{Z_t})
\end{equation*}
for the corresponding isometric embeddings into the common metric space $Z_t$. Since the $\mathbb F$-convergence is time-wise at time $t$, the metric flow pairs converge within the correspondence over the singleton $\{t\}$ in the sense of \cite[Definition~6.1]{Bamler2}. Hence, by \cite[Lemma~6.9]{Bamler2}, the pushed-forward conjugate heat flow measures satisfy $(\phi_t^i)_*\mu_t^i \to (\phi_t^\infty)_*\mu_t^\infty$
in the $W_1$-Wasserstein sense on $Z_t$. Hence, after passing to a subsequence, there exist couplings $q_i\in \mc P(Z_t\times Z_t)$
of $(\phi_t^i)_*\mu_t^i$ and $(\phi_t^\infty)_*\mu_t^\infty$ such that
\begin{equation}\label{eq.coupling}
\int_{Z_t\times Z_t} d_{Z_t}(u,v)\,dq_i(u,v)\to 0.
\end{equation}
By Remark~5.6 of \cite{Bamler2}, time-wise $\mathbb F$-convergence at time $t$ thus implies
\begin{equation*}
d_{GW_1}\Big((\mc X_t^i,d_t^i,\mu_t^i),(\mc X_t^\infty,d_t^\infty,\mu_t^\infty)\Big)\to 0.
\end{equation*}
Since all measures in a metric flow pair have full support, \cite[Lemma~2.13]{Bamler2} shows that, after passing to a subsequence, the embedded supports
\begin{equation*}
\phi_t^i(\mc X_t^i)=\supp\big((\phi_t^i)_*\mu_t^i\big)
\qquad\text{and}\qquad
\phi_t^\infty(\mc X_t^\infty)=\supp\big((\phi_t^\infty)_*\mu_t^\infty\big)
\end{equation*}
converge in the ambient space $Z_t$. In particular, we conclude that we have Gromov--Hausdorff convergence of the time-slices (of a subsequence). 

It remains to show that the chosen $H_n$-centers do not diverge along the limiting procedure. To do so, pick an $H_n$-center $z_t^{\prime\infty}\in \mc X_t^\infty$ of the limit conjugate heat flow $\mu_t^\infty$. Then, following the same lines as in \cite{LW2}, set
\begin{equation*}
r:=\sqrt{10H_n|t|}, \qquad Y_i:=\phi_t^i\big(B_t^i(z_t^i,r)\big), \qquad Y_\infty:=\phi_t^\infty\big(B_t^\infty(z_t^{\prime\infty},r)\big).
\end{equation*}
\textbf{Claim.} There exist points $y_i\in Y_i$ and $y_\infty\in Y_\infty$ such that $y_i\to y_\infty$ in $Z_t$. 
\begin{proof}
Suppose towards a contradiction that the claim is false. Then after passing to a subsequence, there exists $\varepsilon>0$ such that $d_{Z_t}(u,v)\geq \varepsilon$ for all $(u,v)\in Y_i\times Y_\infty$. Then on the one hand, the couplings $q_i$ from \eqref{eq.coupling} satisfy (by the definition of an $H_n$-center and Chebyshev's inequality),
\begin{equation*}
(\phi_t^i)_*\mu_t^i(Z_t\setminus Y_i) = \mu_t^i\big(\mc X_t^i\setminus B_t^i(z_t^i,r)\big) \leq \frac{1}{r^2}\Var(\mu_t^i,\delta_{z_t^i}) \leq \frac{1}{r^2} H_n|t| =\frac{1}{10},
\end{equation*}
and similarly
\begin{equation*}
(\phi_t^\infty)_*\mu_t^\infty(Z_t\setminus Y_\infty) \leq \frac{1}{10},
\end{equation*}
which yields
\begin{align*}
q_i(Y_i\times Y_\infty)
&\geq 1-q_i((Z_t\setminus Y_i)\times Z_t)-q_i(Z_t\times (Z_t\setminus Y_\infty))\\
&=1-(\phi_t^i)_*\mu_t^i(Z_t\setminus Y_i)-(\phi_t^\infty)_*\mu_t^\infty(Z_t\setminus Y_\infty)\\
&\geq 1-\frac{1}{10}-\frac{1}{10}=\frac{4}{5}.
\end{align*}
On the other hand, we find
\begin{equation*}
q_i(Y_i\times Y_\infty) \leq \frac{1}{\varepsilon}\int_{Y_i\times Y_\infty} d_{Z_t}(u,v)\,dq_i(u,v) \leq \frac{1}{\varepsilon}\int_{Z_t\times Z_t} d_{Z_t}(u,v)\,dq_i(u,v) \to 0,
\end{equation*}
yielding the desired contradiction and thus proving the claim.
\end{proof}

Let now $y_i\in Y_i$ and $y_\infty\in Y_\infty$ be as in the claim and define
\begin{equation*}
z_t^{\prime\, i}:=(\phi_t^i)^{-1}(y_i), \qquad z_t^{\prime\prime\infty}:=(\phi_t^\infty)^{-1}(y_\infty).
\end{equation*}
By the convergence of the embedded time-slices in $Z_t$, we obtain pointed Gromov--Hausdorff convergence
\begin{equation*}
(\mc X_t^i,d_t^i,z_t^{\prime\, i})\to (\mc X_t^\infty,d_t^\infty,z_t^{\prime \prime\, \infty}).
\end{equation*}
Moreover, as $y_i\in Y_i$ and $y_\infty\in Y_\infty$, we also have $d_t^i(z_t^i,z_t^{\prime\, i})\leq r$ and $d_t^\infty(z_t^{\prime\infty},z_t^{\prime \prime\, \infty})\leq r$. Thus, we may change basepoints and conclude that there exists a point $z_t^\infty\in \mc X_t^\infty$ such that, after passing to a subsequence,
\begin{equation*}
(\mc X_t^i,d_t^i,z_t^i)\to (\mc X_t^\infty,d_t^\infty,z_t^\infty)
\end{equation*}
in the pointed Gromov--Hausdorff sense, proving the first statement of the lemma.

For the second statement, let $\widetilde z_t^i\in \mc X_t^i$ be $H$-centers of $\mu_t^i$ for some fixed $H\geq H_n$, and let $z_t^i$ be any $H_n$-centers. By the triangle inequality for $d_{W_1}$ and Lemma~2.8 from \cite{Bamler2}, we have
\begin{align*}
d_t^i(\widetilde z_t^i,z_t^i)
&=d_{W_1}^{d_t^i}(\delta_{\widetilde z_t^i},\delta_{z_t^i})\\
&\leq d_{W_1}^{d_t^i}(\delta_{\widetilde z_t^i},\mu_t^i)+d_{W_1}^{d_t^i}(\mu_t^i,\delta_{z_t^i})\\
&\leq \sqrt{\Var(\mu_t^i,\delta_{\widetilde z_t^i})}+\sqrt{\Var(\mu_t^i,\delta_{z_t^i})}\\
&\leq \sqrt{H|t|}+\sqrt{H_n|t|}
\leq 2\sqrt{H|t|}.
\end{align*}
So the $H$-centers stay within a uniformly bounded distance of the corresponding $H_n$-centers, and the same change of basepoint argument as above yields the desired pointed Gromov--Hausdorff convergence for the sequence $(\widetilde z_t^i)$ as well.
\end{proof}

\begin{remark}
One could also infer the Gromov--Hausdorff convergence of the time-slices in the first part of the proof from the smooth convergence on the regular part obtained by Bamler's structure theory, together with the fact that the singular time-slice is the metric completion of the regular part; compare with Theorem~\ref{thm.Bamlerblowup2} and Remark~\ref{remark.smoothconv}.
\end{remark}

\begin{remark}
We only stated the result for tangent flows at finite time singularities or tangent flows at infinity in order to simplify the statement (and since this is sufficient for our purpose). However, one can prove an analogous result for $\mathbb{F}$-convergence of \emph{any} sequence of metric flow pairs $(\mathcal{X}^i,(\mu^i_t))$ to a metric flow pair $(\mathcal{X}^\infty,(\mu^\infty_t))$, assuming that all metric flows are $H$-concentrated for a uniform $H>0$, all $\mu^i_t$ (including $\mu^\infty_t$) are probability measures, and we have Gromov--Hausdorff convergence of a time-slice where the $\mathbb{F}$-convergence is time-wise. In this case, one obtains pointed Gromov--Hausdorff convergence with respect to $H'$-centers ($H' \geq H$) in that slice.
\end{remark}

We remark that by Lemma~6.3 in \cite{Bamler2}, the $\mathbb{F}$-convergence is certainly time-wise outside a set of measure zero. However, for our intended application, it is important to apply Lemma~\ref{GH-conv} at any time-slice; otherwise we would only obtain that the scalar curvature blows up at a Type~I rate along a sequence of times. For this reason, we need to show that the $\mathbb{F}$-convergence is time-wise at \emph{all} times $t$. By Theorem~7.6 in \cite{Bamler2}, the convergence is time-wise at any time at which $\mathcal{X}^\infty$ is continuous. Moreover, it is uniform on any interval of times consisting only of continuity times for $\mathcal{X}^\infty$. (Notice that in these results we do not care about the conjugate heat kernel $(\mu_t^\infty)$). The following lemma shows that in our case the limit $\mathcal{X}^\infty$ is continuous at all times.

\begin{lemma}[Continuity of the tangent flow]\label{lemma.continuity}
In the situation described in Theorems~\ref{thm.Bamlerblowup1} and \ref{thm.Bamlerinfinity}, the tangent metric flow $\mathcal{X}^\infty$ is continuous at any time $t \in (-\infty,0)$ and therefore the $\mathbb{F}$-convergence is time-wise for any $t \in (-\infty,0)$.
\end{lemma}

\begin{proof}
From one of the equivalent statements in \cite[Theorem~4.9]{Bamler2}, it suffices, for any $t_0 \in (-\infty,0)$, to find a conjugate heat flow $(\mu_t)$ with full support and finite variance, defined around $t_0$, such that the function
\begin{equation}
t \mapsto \int_{\mathcal{X}^\infty_t} \int_{\mathcal{X}^\infty_t} d_t(x,y) \, d\mu_t(x)\, d\mu_t(y)
\end{equation}
is continuous at time $t_0$. Choose $(\mu_t)$ to be a conjugate heat flow $(\nu_{(x,0);t})$ based at some point $x \in \mathcal{X}^\infty_0$. By the precise structure of $\mathcal{X}^\infty$ from Theorem~\ref{thm.Bamlerblowup2}, under the identification $\mathcal{X}_{<0}^\infty \cong X \times (-\infty,0)$, the conjugate heat flow based at $x\in \mathcal{X}_0^\infty$ is time-independent, $\nu_{(x,0);t}=\mu$, and the slice metrics satisfy $d_t=\sqrt{|t|}\,d$.
Therefore
\begin{equation}
t \mapsto \int_{X} \int_{X} \sqrt{|t|}\, d(x,y)\, d\mu(x)\, d\mu(y),
\end{equation}
which is clearly continuous at any $t_0 \in (-\infty,0)$.
\end{proof}

We have so far obtained pointed Gromov-Hausdorff convergence based at $H$-centers for \emph{all} times $t$. The next lemma shows that we can use a \emph{fixed Type~I singular point $p$} as the base-point for the pointed Gromov-Hausdorff convergence in Lemma~\ref{GH-conv}.

\begin{lemma}[$H$-center]\label{lemma.Hcenter}
If $p \in \Sigma_I$ is a Type~I point, then there exist a positive constant $H$ depending only on the Type~I constant $C_I(p)$ from \eqref{eq.typeI} and the dimension $n$, and a time $T_0 \in [0,T)$ depending on $p$, such that the following holds:
for all $s\in (T_0,T)$, the point $(p,s)$ is an $H$-center of $(p,T)$ in the sense that
\begin{equation*}
\Var\bigl(\widetilde{\nu}_{(p,T);s},\delta_{(p,s)}\bigr)  \leq H(T-s).
\end{equation*}
\end{lemma}

\begin{proof}
We can adopt the proof from the Ricci shrinker case by the first author with Bertellotti (see Proposition~6.1 in \cite{BB25}), which itself is an extension of a result by Li-Wang for regular heat kernels (Proposition 4.4 in \cite{LW2}) to heat kernels based at the singular time.

Fix $s<T$ and let $t_i \nearrow T$ be a sequence defining the singular-time conjugate heat kernel
based at $(p,T)$, so that $\nu_{(p,t_i);s} \to \widetilde{\nu}_{(p,T);s}$ in $W_1$ as explained at the beginning of this section. Since $p \in \Sigma_I$, there is a constant $C<\infty$ (depending only on $C_I(p)$ and $n$) such that, after possibly shrinking the time interval towards $T$, we have
\begin{equation*}
R(p,\tau)\leq \frac{C}{T-\tau} \leq \frac{C}{t_i-\tau} , \qquad \forall \tau\in [s,t_i).
\end{equation*}
For every $i$ large enough, we estimate Perelman's reduced length from $(p,t_i)$ to $(p,s)$ by means
of the constant path $\gamma(\tau)\equiv p$, obtaining
\begin{equation*}
\ell_{(p,t_i)}(p,s) \leq \frac{1}{2\sqrt{t_i-s}}\int_s^{t_i}\sqrt{t_i-\tau}\,R(p,\tau)\,d\tau \leq C.
\end{equation*}
Hence, writing
\begin{equation*}
d\nu_{(p,t_i);s}=K(p,t_i;\cdot,s)\,dV_{g(s)}
=(4\pi (t_i-s))^{-\frac n2}e^{-f_i(\cdot,s)}\,dV_{g(s)},
\end{equation*}
Perelman's differential Harnack inequality for the conjugate heat kernel
\cite[Section~9]{Per} yields
\begin{equation*}
f_i(p,s)\leq \ell_{(p,t_i)}(p,s)\leq C.
\end{equation*}
Therefore we obtain the on-diagonal lower bound
\begin{equation*}
K(p,t_i;p,s)\geq e^{-C}(4\pi (t_i-s))^{-\frac n2}.
\end{equation*}
By the same argument as in Proposition~4.4 of \cite{LW2}, such an on-diagonal lower bound, together with Bamler's heat kernel upper bound from Theorem 7.2 in \cite{Bamler1}, implies that
$(p,s)$ is an $H_0$-center of $(p,t_i)$ for some constant
$H_0=H_0(n,C_I)$ independent of $i$, that is,
\begin{equation*}
\Var_s\!\big(\nu_{(p,t_i);s},\delta_p\big)\leq H_0\,(t_i-s).
\end{equation*}

We now pass to the singular-time heat kernel. By the triangle inequality for $\sqrt{\Var_s}$,
\cite[Lemmas~2.8 and 2.9]{Bamler2}, and the uniform variance bounds
\begin{equation*}
\Var_s\!\big(\nu_{(p,t_i);s}\big)\leq H_n\,(t_i-s), \qquad \Var_s\!\big(\widetilde{\nu}_{(p,T);s}\big)\leq H_n\,(T-s),
\end{equation*}
we obtain
\begin{align*}
\Var_s\!\big(\widetilde{\nu}_{(p,T);s},\delta_p\big) &\leq 2\Var_s\!\big(\nu_{(p,t_i);s},\delta_p\big) +2 \Var_s\!\big(\nu_{(p,t_i);s},\widetilde{\nu}_{(p,T);s}\big) \\
&\leq 2H_0\,(t_i-s) +4 d_{W_1}^{g(s)}\!\big(\nu_{(p,t_i);s},\widetilde{\nu}_{(p,T);s}\big)^2 + 4 \Var_s\!\big(\nu_{(p,t_i);s}\big) + 4\Var_s\!\big(\widetilde{\nu}_{(p,T);s}\big) \\
&\leq (2H_0+8H_n)(T-s) +4 d_{W_1}^{g(s)}\!\big(\nu_{(p,t_i);s},\widetilde{\nu}_{(p,T);s}\big)^2.
\end{align*}
Letting $i\to\infty$, the last term vanishes and we obtain the lemma for $H:=2H_0+8H_n$.
\end{proof}

The last ingredient needed for the proof of our main theorem is the following rigidity result.

\begin{lemma}[Tangent metric cones are Ricci flat]\label{lemma.Ricciflat}
In the situation described above, suppose the scalar curvature $\Sc^{\infty}$ vanishes at a point $x_0$ in the regular part $\mathcal{R}$ of $\mathcal{X}^\infty$, that is $\Sc^{\infty}(x_0)=0$. Then $\Ric \equiv 0$ on $\mathcal{R}$, hence the tangent flow $\mathcal{X}^\infty$ is static and Theorem~2.16 from \cite{Bamler3} applies to it.
\end{lemma}

\begin{proof}
Let $x_0$ be a point in $\mathcal{R}$ where $\Sc^{\infty}(x_0)=0$. Because $\mathcal{R}$ is open, there exists a radius $r_0$ such that the $*$-parabolic ball $\mathcal{P}^*(x_0;r_0) \subset \mathcal{R}$ (these form a basis of the natural topology by \cite{Bamler2}). In particular, the curvature is bounded there, and we can fit a standard parabolic cylinder $\mathcal{P}(x_0;r_1) \subset \mathcal{P}^*(x_0;r_0) \subset \mathcal{R}$. Therefore, as $\mathcal{P}(x_0;r_1)$ is a product domain, we may restrict our attention to the time-slice $t_0 := \mathfrak{t}(x_0)$ and write $x_0 = (p_0, t_0)$. By the soliton equation (which holds on the regular part $\mathcal{R}$), adapting the classical argument of Zhang~\cite[Proposition~2.2]{Zh09}, we have
\begin{equation}\label{eq.evolR}
\Delta \Sc^\infty = \langle \nabla f, \nabla \Sc^\infty \rangle + \Sc^\infty - 2|\Ric^\infty|^2 \leq \langle \nabla f, \nabla \Sc^\infty \rangle + \Sc^\infty.
\end{equation}
Since $\Sc^\infty \geq 0$ and $\Sc^\infty(p_0, t_0) = 0$, the strong minimum principle implies that $\Sc^\infty(\cdot, t_0) \equiv 0$ on $B_{t_0}(x_0, r_1)$.

Since the flow is smooth on $\mathcal{R}$ and has bounded curvature on the above parabolic cylinder, Corollary~1.3 of \cite{Kot13} implies that the metric $g^\infty_{t_0}$ is real-analytic on $\mathcal{R}_{t_0}$. Consequently, the scalar curvature $\Sc^\infty(\cdot,t_0)$ is a real-analytic function on $\mathcal{R}_{t_0}$ and because it vanishes on the non-empty open subset $B_{t_0}(p_0,r_1)$, it must vanish on the entire connected component of $\mathcal{R}_{t_0}$ containing $x_0$ by the identity theorem for real-analytic functions.

Now using the structure from Theorem~\ref{thm.Bamlerblowup2}, in particular the identification
\begin{equation*}
\mathcal{R}=\mathcal{R}_X\times (-\infty,0), \qquad (\mathcal{R}_t,g^\infty_t)=(\mathcal{R}_X\times\{t\},|t|g_X).
\end{equation*}
we conclude that $\mathcal{R}_{t_0}=\mathcal{R}_X \times \set{t_0}$ is connected, so the scalar curvature vanishes on the entire $\mathcal{R}_{t_0}$. Furthermore, the self-similarity of the flow implies vanishing of the scalar curvature on all of $\mathcal{R}$. Plugging this back into \eqref{eq.evolR}, we conclude
\begin{equation}
0 = \Delta_f \Sc^\infty = -2|\Ric^\infty|^2,
\end{equation}
hence $\Ric^\infty \equiv 0$ on $\mathcal{R}$. This finishes the proof.
\end{proof}

\begin{remark}
Notice that we cannot directly adapt Zhang's classical proof that establishes $\Sc^\infty \geq 0$ on Ricci shrinkers from equation \eqref{eq.evolR}, because his argument relies on cutting off the scalar curvature on larger and larger regions while remaining in the regular part of the shrinker. However, since non-negativity of the scalar curvature is guaranteed by Theorem~\ref{thm.Bamlerblowup2}, we can argue in a fixed local neighbourhood as shown above.

Also note that the lack of any global bound on the curvature in the assumptions of Corollary 1.3 of \cite{Kot13} is crucial in the proof above, as no such bound can exist for a general singular time-slice $\mathcal{X}_{t_0}$. Compare also with Theorem 1 in \cite{Kot15}. 

We conjecture that any $\mathbb{F}$-limit of smooth complete Ricci flows satisfying \eqref{eq.bddcurv} is actually space-time analytic on the regular part $\mathcal{R}$, but this conjecture lies beyond the scope of this article. Such a space-time analyticity result would rule out wild (but smooth) behaviours of the Ricci flow on the regular part $\mathcal{R}$. In principle, without space-time analyticity a smooth Ricci flow in consideration, may be Ricci-flat (or even completely flat) on a ball $B_{t_0}(x_0, r_1)$ as considered above, but evolve non-trivially at later times as described for example in the beautiful work of Topping \cite{Top10}. Luckily, self-similarity comes in to save us from such a behaviour in our case.
\end{remark}

We are now ready to prove our first main result that the scalar curvature blows up at any Type~I singular point.

\begin{proof}[Proof of Theorem~\ref{mainthm}]
Assume towards a contradiction that there exists an increasing sequence of times $t_i \nearrow T$
such that
\begin{equation}\label{eq.scalar-small-seq}
(T-t_i)\Sc(p,t_i)\to 0.
\end{equation}
Set $\tau_i:=T-t_i$ and define the rescaled flows
\begin{equation*}
g^i(t):=\tau_i^{-1}g(T+\tau_i t), \qquad t\in[-\tau_i^{-1}T,0).
\end{equation*}
Moreover, let $\mu_t^i:=\widetilde{\nu}_{(p,T);T+\tau_i t}$, so that $(M,(g^i(t))_{[-\tau_i^{-1}T,0)},(\mu_t^i)_{[-\tau_i^{-1}T,0)})$ is a sequence of metric flow pairs as in Theorem~\ref{thm.Bamlerblowup1}. After passing to a subsequence, we may assume that these metric flow pairs $\mathbb{F}$-converge within a correspondence $\mathfrak C$ to a tangent flow $(\mathcal{X}^\infty,(\mu_t^\infty)_{t\in(-\infty,0)})$. By Lemma~\ref{lemma.continuity}, the tangent flow $\mathcal{X}^\infty$ is continuous at every time $t<0$, and hence the $\mathbb{F}$-convergence is time-wise at every time $t<0$.

In the following we \textit{fix} some time $t_0 \in [-1,0)$. By Lemma~\ref{lemma.Hcenter}, there exists $H>0$ depending only on $C_I(p)$ and $n$ such that $(p,t_0)$ is an $H$-center of $(p,0)$ in every rescaled flow $(M,g^i(t))$ for all $i$ sufficiently large. Applying Lemma~\ref{GH-conv} at time $t_0$, we therefore obtain, after passing to a further subsequence, pointed Gromov--Hausdorff convergence
\begin{equation}\label{eq.pGH-main}
(M,d_{g^i(t_0)},p)\to (\mathcal{X}_{t_0}^\infty,d_{t_0}^\infty,z_{t_0}^\infty)
\end{equation}
for some point $z_{t_0}^\infty\in \mathcal{X}_{t_0}^\infty$.

Since $p\in \Sigma_I$, \eqref{eq.singularstrong} and \eqref{eq.typeI} imply that there exists a constant $1 \leq C_I<\infty$ such that we have $1 \leq (T-t)\, r_{\Rm}^{-2}(p,t)\leq  2 C_I$ for all $t$ sufficiently close to $T$, say $t\in [\widetilde{T},T)$. For the rescaled flows, this means that
\begin{equation}\label{eq.singularlowerupper}
\frac{1}{|t|} \leq r_{\Rm,g^i}^{-2}(p,t) \leq \frac{2C_I}{|t|},
\end{equation}
for all $t\in [(\widetilde{T}-T)\tau_i^{-1},0)$. Note that $(\widetilde{T}-T)\tau_i^{-1} \to -\infty$ and therefore \eqref{eq.singularlowerupper} holds in particular for all $t \in [-1,0)$ whenever $i$ is sufficiently large. In particular, for our $t_0$, we find for $i$ sufficiently large,
\begin{equation*}
r_{\Rm,g^i}(p,t_0)\geq \Big( \frac{2 C_I}{|t_0|} \Big)^{-1/2}>0.
\end{equation*}
Hence the basepoints stay in a uniform size region with uniform curvature bounds, so the pointed Gromov--Hausdorff convergence at time $t_0$ upgrades to smooth Cheeger--Gromov convergence in a neighbourhood of the limit basepoint $z^{\infty}_{t_0}\in \mathcal{X}_{t_0}^\infty$. In particular, $z^{\infty}_{t_0}$ lies in the regular part $\mathcal{R}_{t_0}$ of the limit.

Let us now consider $t_0=-1$. By \eqref{eq.scalar-small-seq}, we have $\Sc_{g^i}(p,-1)=\tau_i \Sc(p,t_i)\to 0$. By the smooth convergence around $z_{-1}^\infty$, this implies $\Sc^\infty(z_{-1}^\infty)=0$. Hence, by Lemma \ref{lemma.Ricciflat}, we get $\Ric^\infty\equiv 0$, and therefore the tangent flow $\mathcal{X}^\infty$ is static and Theorem~2.16 of \cite{Bamler3} applies. 
This means that we have an identification $\mathcal{X}^\infty_{<0} = X \times (-\infty,0)$ as well as 
\begin{equation}\label{eq.identification216singular}
\mathcal{R} = \mathcal{R}_X \times (-\infty,0), \qquad (\mathcal{R}_t,g_t^{\infty})=(\mathcal{R}_X,g_X),
\end{equation} 
for a fixed time-independent Riemannian metric $g_X$ on the regular factor. In particular, the spacetime convergence in Remark~\ref{remark.smoothconv} is time-preserving on the regular part, so the limit point corresponding to the base point $p$ is represented by a single spatial point $z_0\in R_X$ for every time slice. In other words, there exists $z_0 \in \mathcal{R}_X$ such that $z^\infty_{t_0} = (z_0,t_0)$ under the identification \eqref{eq.identification216singular} for all $t_0\in [-1,0)$.

The smooth Cheeger-Gromov convergence allows us to pass \eqref{eq.singularlowerupper} to the limit to conclude that
\begin{equation}\label{eq.contradictionestsingular}
\frac{1}{|t_0|} \leq r_{\Rm^\infty}^{-2}(z_0,t_0) \leq \frac{2C_I}{|t_0|},
\end{equation}
for all $t_0 \in [-1,0)$. But because $g_t^\infty=g_X$ is independent of $t$, the quantity $r_{\Rm^\infty}^{-2}(z_0,t_0)$ is independent of $t_0$ as well, which is impossible since the left-hand side of \eqref{eq.contradictionestsingular} blows up as $t_0\to 0$ while the right-hand side guarantees the finiteness of this constant by fixing any $t_0\in [-1,0)$. This contradiction completes the proof.
\end{proof}

\section{Ancient Ricci Flows}\label{sec.ancient}

Let us now focus on ancient Ricci flows $(M,g(t))_{t\in(-\infty,0]}$ that are complete at all times and have bounded curvature on every compact time interval as in \eqref{eq.bddcurv}. We start by proving Proposition~\ref{prop.ancient}.

\begin{proof}[Proof of Proposition \ref{prop.ancient}]
Suppose towards a contradiction that for some $p\in M$ and a sequence $t_j\rightarrow -\infty$ we have $|t_j| r_{\Rm}^{-2}(p,t_j)\leq 1$. This means that we have the Riemann curvature bound 
\begin{equation}\label{eq.110Rmbound}
|\Rm|(q,s) \leq \frac{1}{|t_j|}, \qquad \forall (q,s)\in B_{g(t_j)}\Big(p,\sqrt{\abs{t_j}}\Big) \times (2 t_j,0).
\end{equation}
We now want to place a ball measured with the metric $g(-1)$ into this cylinder. To do so, note that \eqref{eq.110Rmbound} implies that 
\begin{equation}\label{eq.110Ricbound}
\Ric (q,s) \leq \frac{(n-1)}{|t_j|} g(q,s), \qquad \forall (q,s)\in B_{g(t_j)}\Big(p,\sqrt{\abs{t_j}}\Big) \times (2 t_j,0).
\end{equation}
By the standard distance distortion estimates, we therefore must have
\begin{equation}
B_{g(t_j)}\Big(p,\sqrt{\abs{t_j}}\Big) \supseteq B_{g(-1)}\Big(p,\sqrt{\abs{t_j}} e^{-\int_{t_j}^{-1} \tfrac{n-1}{\abs{t_j}} ds }\Big)=B_{g(-1)}\Big(p,c\sqrt{\abs{t_j}}\Big),
\end{equation}
where we have set $c:= e^{-(n-1)}>0$. Hence, on $B_{g(-1)}(p,c\sqrt{|t_j|}) \times \{-1 \}$ we have $\abs{\Rm (\cdot,-1)}\le\tfrac{1}{|t_j|}$, and letting $j \rightarrow +\infty$ we deduce that $(M,g(-1))$ is flat. This gives the flatness of the entire flow by uniqueness in the category of complete Ricci flows with bounded curvature on compact time intervals, and hence the desired contradiction. Since the proof relied only on classical distance distortion estimates, available once we have the Ricci bound \eqref{eq.110Ricbound} that also follows from a bound on the Ricci scale, the same argument can be adapted to prove the statement for the Ricci scale instead.
\end{proof}

Next, we show an analogue of Lemma \ref{lemma.Hcenter} adapted to ancient Type~I points.

\begin{lemma}[$H$-center:~Ancient case]\label{lemma.Hcenterancient}
If $p \in A_I$ is an ancient Type~I point, then there exist a positive constant $H$ and a time $T_0 \in (-\infty,0)$, such that the following holds: for all $s\in (-\infty,T_0]$, the point $(p,s)$ is an $H$-center of $(p,0)$ in the sense that
\begin{equation*}
\Var_s\!\big(\nu_{(p,0);s},\delta_p\big)\leq H\,|s|.
\end{equation*}
Here $H$ may depend on $n$, $C_I(p)$, $T_0$, as well as a bound for $\Sc(p,t)$ on the compact time interval $[T_0,0]$, but it is independent of $s\leq T_0$.
\end{lemma}

\begin{proof}
The proof is a simple adaptation of the one of Lemma \ref{lemma.Hcenter} given before, but here we can directly work with conjugate heat kernels based at the regular time $0$. Since $p \in A_I$, there exist a constant $C_1<\infty$ (depending only on $C_I(p)$ and $n$) and a time $T_0 \in (-\infty,0)$, such that we have
\begin{equation*}
R(p,t)\leq \frac{C_1}{|t|}, \qquad \forall t\in (-\infty,T_0].
\end{equation*}
Moreover, by assumption \eqref{eq.bddcurv}, there exists $C_2<\infty$ (depending on $T_0$) such that
\begin{equation*}
R(p,t)\leq C_2, \qquad \forall t\in [T_0,0].
\end{equation*}
Combining these two bounds, we can estimate Perelman's reduced length from $(p,0)$ to $(p,s)$ by means of the constant path $\gamma(\tau)\equiv p$, obtaining
\begin{align*}
\ell_{(p,0)}(p,s) &\leq \frac{1}{2\sqrt{|s|}}\int_s^{0}\sqrt{{-\tau}}\,R(p,\tau)\,d\tau \\
&\leq \frac{C_1}{2\sqrt{|s|}}\int_s^{T_0}\frac{1}{\sqrt{|\tau|}}\,d\tau + \frac{C_2}{2\sqrt{|s|}}\int_{T_0}^0\sqrt{|\tau|}\,d\tau \\
&= \frac{C_1}{\sqrt{|s|}}(\sqrt{\abs{s}}-\sqrt{\abs{T_0}}) + \frac{C_2}{3\sqrt{|s|}}\abs{T_0}^{3/2} \\
&\leq C_1 + \frac{C_2}{3}\abs{T_0} =: C.
\end{align*}
Note that this constant depends on $C_1$, $C_2$, and $T_0$, but is independent of $s \leq T_0$. Hence, writing
\begin{equation*}
d\nu_{(p,0);s}=K(p,0;\cdot,s)\,dV_{g(s)}
=(4\pi |s|)^{-\frac n2}e^{-f(\cdot,s)}\,dV_{g(s)},
\end{equation*}
Perelman's differential Harnack inequality for the conjugate heat kernel
\cite[Section~9]{Per} yields
\begin{equation*}
f(p,s)\leq \ell_{(p,0)}(p,s)\leq C,
\end{equation*}
giving once again an on-diagonal lower bound of the form
\begin{equation*}
K(p,0;p,s)\geq e^{-C}(4\pi |s|)^{-\frac n2}.
\end{equation*}
As before, by the argument in Proposition~4.4 of \cite{LW2}, such an on-diagonal lower bound, together with Bamler's heat kernel upper bound from Theorem 7.2 in \cite{Bamler1}, implies that
$(p,s)$ is an $H$-center of $(p,0)$ for some constant
$H=H(n,C_1,C_2,T_0)$, that is,
\begin{equation*}
\Var_s\!\big(\nu_{(p,0);s},\delta_p\big)\leq H\,|s|.
\end{equation*}
This finishes the proof.
\end{proof}

We are ready to prove our second main result, Theorem \ref{ancientthm}, regarding the scalar curvature lower bound at ancient Type I points.

\begin{proof}[Proof of Theorem \ref{ancientthm}]
Assume towards a contradiction that there exists a decreasing sequence of times $(t_i)$ diverging to $-\infty$ such that
\begin{equation}\label{eq.scalar-small-ancient}
|t_i|\Sc(p,t_i) \rightarrow 0.
\end{equation}
Set $\tau_i:=|t_i|$ and define the rescaled flows
\begin{equation}
g^i(t):=\tau_i^{-1}g(\tau_i t), \quad t \in (-\infty,0].
\end{equation}
Moreover, let $\mu_t^i:=\nu_{(p,0);\tau_i t}$, so that $(M,(g^i(t))_{(-\infty,0]},(\mu_t^i)_{(-\infty,0]})$ is a sequence of metric flow pairs as in Theorem~\ref{thm.Bamlerinfinity}. After passing to a subsequence, we may assume that these metric flow pairs $\mathbb{F}$-converge within a correspondence $\mathfrak C$ to a tangent flow at infinity $(\mathcal{X}^\infty,(\mu_t^\infty)_{t\in(-\infty,0]})$. By Lemma~\ref{lemma.continuity}, the tangent flow $\mathcal{X}^\infty$ is continuous at every time $t<0$, and hence the $\mathbb{F}$-convergence is time-wise at every time $t<0$. 

In the following, we \emph{fix} some $t_0\leq -1$. By Lemma~\ref{lemma.Hcenterancient}, there exists $H>0$ such that $(p,t_0)$ is an $H$-center of $(p,0)$ in every rescaled flow $(M,g^i(t))$ for all $i$ sufficiently large (and with $H$ independent of $i$). Applying Lemma~\ref{GH-conv} at time $t_0$, we therefore obtain, after passing to a further subsequence, pointed Gromov--Hausdorff convergence
\begin{equation}\label{eq.pGH-ancient}
(M,d_{g^i(t_0)},p)\to (\mathcal{X}_{t_0}^\infty,d_{t_0}^\infty,z_{t_0}^\infty)
\end{equation}
for some point $z_{t_0}^\infty\in \mathcal{X}_{t_0}^\infty$.

Since $p\in A_I$, \eqref{eq.ancienttypeI} and \eqref{eq.strongancient} imply that there exists a constant $1 \leq C_I<\infty$ such that we have $1 \leq |t|\, r_{\Rm}^{-2}(p,t)\leq 2C_I$ for all $t$ negative enough, say $t\leq \widetilde{T}$. For the rescaled flows, this means that
\begin{equation}\label{eq.111lowerupper}
\frac{1}{|t|} \leq r_{\Rm,g^i}^{-2}(p,t) \leq \frac{2C_I}{|t|},
\end{equation}
for all $t\leq \abs{t_i}^{-1}\widetilde{T}$. Note that $\abs{t_i}^{-1}\widetilde{T} \to 0$ and therefore \eqref{eq.111lowerupper} holds in particular for all $t \leq -1$ whenever $i$ is sufficiently large. In particular, for our $t_0 \leq 1$, we find for $i$ sufficiently large
\begin{equation*}
r_{\Rm,g^i}(p,t_0)\geq \Big(\frac{2C_I}{\abs{t_0}}\Big)^{-1/2}>0.
\end{equation*}
Hence the basepoints $(p,t_0)$ stay in a uniform size region with uniform curvature bounds, so the pointed Gromov--Hausdorff convergence at time $t_0$ upgrades to smooth Cheeger--Gromov convergence in a neighbourhood of the limit basepoint $z^{\infty}_{t_0}\in \mathcal{X}_{t_0}^\infty$. In particular, $z^{\infty}_{t_0}$ lies in the regular part $\mathcal{R}_{t_0}$ of the limit.

Let us now consider $t_0=-1$. By \eqref{eq.scalar-small-ancient}, we have $\Sc_{g^i}(p,-1)=\tau_i \Sc(p,t_i)\to 0$. By the smooth convergence around $z_{-1}^\infty$, this implies $\Sc^\infty(z_{-1}^\infty)=0$ and hence, by Lemma~\ref{lemma.Ricciflat}, we get $\Ric^\infty\equiv 0$ on the regular part. Consequently, the tangent flow $\mathcal{X}^\infty$ is static and Theorem~2.16 of \cite{Bamler3} applies. This means that we have an identification $\mathcal{X}^\infty_{<0} = X \times (-\infty,0)$ as well as 
\begin{equation}\label{eq.identification216}
\mathcal{R} = \mathcal{R}_X \times (-\infty,0), \qquad (\mathcal{R}_t,g_t^{\infty})=(\mathcal{R}_X,g_X),
\end{equation} 
for a fixed time-independent Riemannian metric $g_X$ on the regular factor. In particular, the spacetime convergence in Remark~\ref{remark.smoothconv} is time-preserving on the regular part, so the limit point corresponding to the base point $p$ is represented by a single spatial point $z_0\in R_X$ for every time slice. In other words, there exists $z_0 \in \mathcal{R}_X$ such that $z^\infty_{t_0} = (z_0,t_0)$ under the identification \eqref{eq.identification216} for all $t_0\leq -1$.

The smooth Cheeger-Gromov convergence allows us to pass \eqref{eq.111lowerupper} to the limit to conclude that
\begin{equation}\label{eq.contradictionest}
\frac{1}{|t_0|} \leq r_{\Rm^\infty}^{-2}(z_0,t_0) \leq \frac{2C_I}{|t_0|},
\end{equation}
for all $t_0 \leq -1$. But because $g_t^\infty=g_X$ is independent of $t$, the quantity $r_{\Rm^\infty}^{-2}(z_0,t_0)$ is independent of $t_0$ as well, which is impossible since the right-hand side of \eqref{eq.contradictionest} tends to $0$ as $t_0\to -\infty$ while the left-hand side is bounded below by the positive constant after fixing any $t_0\leq -1$. This contradiction completes the proof.
\end{proof}

\begin{remark}
Let us remark that, in the complete and bounded curvature case, the only possibility for a shrinking soliton to be Ricci flat, is for it to be Riemann flat, and actually the Gaussian soliton. Dropping one of the two assumptions, so for example considering singular Ricci solitons as those appearing in the proof above, we may have Ricci flatness without the soliton being Gaussian.

For example, if one considers a non-flat Einstein manifold $(L^{n-1},g_L)$, with $\Ric_L=(n-2)g_L$, then the singular cone $(C(L),g)$, where $g:= dr^2 + r^2 g_L$, and $r$ is the distance from the tip, is Ricci flat, non flat, and admits a shrinking soliton structure with potential $f= r^2/4$.
\end{remark}

\makeatletter
\def\@listi{%
  \itemsep=0pt
  \parsep=1pt
  \topsep=1pt}
\makeatother
{\fontsize{10}{11}\selectfont

}
\printaddress

\end{document}